\documentclass{article}
\usepackage[utf8]{inputenc}

\usepackage{amsmath,amsfonts,amssymb,amsthm,fancyhdr,bm,verbatim,hyperref,centernot}
\usepackage{fullpage}
\usepackage{mathtools}
\usepackage{todonotes}
\usepackage{amssymb}

\usepackage{tikz}
\usepackage{subcaption}
\tikzstyle{vertex}=[circle,fill=black,inner sep=2pt]

\makeatletter
\numberwithin{equation}{section}
\numberwithin{figure}{section}
\theoremstyle{plain}
\newtheorem{thm}{\protect\theoremname}
\theoremstyle{plain}

\theoremstyle{plain}
\newtheorem{lem}[thm]{\protect\lemmaname}

\numberwithin{thm}{section}
\theoremstyle{definition}

\newcommand{\eps}{\varepsilon}

\makeatother

\usepackage{babel}
\providecommand{\corollaryname}{Corollary}
\providecommand{\lemmaname}{Lemma}
\providecommand{\theoremname}{Theorem}
\providecommand{\conjecturename}{Conjecture}
\providecommand{\propositionname}{Proposition}
\providecommand{\problemname}{Problem}

\title{Set-coloring Ramsey numbers and error-correcting codes\\ near the zero-rate threshold}
\author{David Conlon\thanks{Department of Mathematics, California Institute of Technology, Pasadena, CA 91125. Email: {\tt dconlon@caltech.edu}. Research supported by NSF Award DMS-2054452.} \and
Jacob Fox\thanks{Department of Mathematics, Stanford University, Stanford, CA 94305. Email: {\tt jacobfox@stanford.edu}. Research supported by a Packard Fellowship and by NSF Awards DMS-1953990 and DMS-2154129.} \and
Huy Tuan Pham\thanks{Department of Mathematics, Stanford University, Stanford, CA 94305. Email: {\tt huypham@stanford.edu}. Research supported by a Two Sigma Fellowship.}
\and Yufei Zhao\thanks{Department of Mathematics, Massachusetts Institute of Technology, Cambridge, MA 02139. Email: {\tt yufeiz@mit.edu}. Research supported by NSF CAREER Award DMS-2044606, a Sloan Research Fellowship and the MIT Solomon
Buchsbaum Fund.}}
\date{}

\begin{document}

\maketitle

\begin{abstract} 
For positive integers $n,r,s$ with $r > s$, the set-coloring Ramsey number $R(n;r,s)$ is the minimum $N$ such that if every edge of the complete graph $K_N$ receives a set of $s$ colors from a palette of $r$ colors, then there is a subset of $n$ vertices where all of the edges between them receive a common color. 
If $n$ is fixed and $\frac{s}{r}$ is less than and bounded away from $1-\frac{1}{n-1}$, then $R(n;r,s)$ is known to grow exponentially in $r$, while if $\frac{s}{r}$ is greater than and bounded away from $1-\frac{1}{n-1}$, then $R(n;r,s)$ is bounded. Here we prove bounds for $R(n;r,s)$ in the intermediate range where $\frac{s}{r}$ is close to $1 - \frac{1}{n-1}$ by establishing a connection to the maximum size of error-correcting codes near the zero-rate threshold. 
\end{abstract}

\section{Introduction}

Two of the central problems in discrete mathematics are that of estimating the maximum size of error-correcting codes with given parameters and that of estimating Ramsey numbers. Here, building on recent work by an overlapping set of authors~\cite{CFHMSV}, we find a close connection between these two problems. More precisely, we show that the problem of estimating set-coloring Ramsey numbers, a natural generalization of the usual Ramsey numbers, and that of estimating the size of error-correcting codes near the zero-rate threshold are essentially the same problem. 

To say more, let $A_q(m,d)$ be the maximum size of a code $C \subseteq [q]^m$ of length $m$ in which any two codewords have Hamming distance at least $d$, i.e., they differ in at least $d$ coordinates. Such a code is called a $q$-ary code of length $m$ and distance $d$. The rate of the code is then defined as $(\log_q |C|)/m$. A result going back to work of Plotkin~\cite{Plot}, who treated the binary case, says that there are codes of positive rate, that is, with exponentially many elements, if $d < (1 - 1/q - \epsilon)m$ for any fixed $\epsilon > 0$ and no such codes if $d \geq (1-1/q)m$. That is, there is a threshold at distance $(1-1/q)m$ where the rate becomes zero. 

On the other hand, for any positive integers $n, r, s$ with $r > s$, we define the set-coloring Ramsey number $R(n; r, s)$ to be  
the minimum $N$ such that if every edge of $K_N$ receives a set of $s$ colors from a palette of $r$ colors, then there is guaranteed to be a monochromatic clique on $n$ vertices, that is, a copy of $K_n$ whose edges all share a common color. As a shorthand, it will be convenient for us to refer to such a set-coloring as an $(r, s)$-coloring of $K_N$. 

A priori, it is not clear that these quantities should have anything to do with one another. However, in~\cite{CFHMSV}, it was shown how to use the Gilbert--Varshamov bound, a standard lower bound for the size of codes, to show that for any $\eps > 0$ there exists $c > 0$ such that $R(n; r, s) \geq 2^{c rn}$ for any $r$ and $s$ with $\eps r < s < (1 - \eps)r$ and $n$ sufficiently large in terms of $r$, a result which is tight up to the constant $c$ (see also~\cite{ACMMM23} for an alternative approach with an improved bound on the constant $c$ in terms of $\eps$). Moreover, the following result was noted.

\begin{thm}[\cite{CFHMSV}] \label{thm:ramcode1}
For all positive integers $q, r, s$ with $r > s$, $R(q+1; r, s) \geq A_{q}(r, s) + 1$.
\end{thm}

In particular, if $q$ is fixed, we see that, provided $s/r \leq 1 - 1/q - \eps$ for some fixed $\eps > 0$, $R(q+1; r, s)$ grows at least exponentially in $r$. Moreover, it was also shown in~\cite{CFHMSV} that if $s/r \geq 1 - 1/q + \eps$ for some fixed $\eps > 0$, then $R(q+1; r, s)$ is at most a constant depending only on $q$ and $\eps$. That is, for $q$ fixed, there is a threshold for $s/r$ at $1 - 1/q$ where the set-coloring Ramsey number $R(q+1; r, s)$ goes from growing exponentially in $r$ to being bounded.

In~\cite{CFHMSV}, it was suggested that perhaps Theorem~\ref{thm:ramcode1} is almost tight when $s/r$ is close to $1 - 1/q$. That this is indeed the case is our first new result.

\begin{thm} \label{thm:ramcode2}
For any positive integer $q$ and any $\epsilon>0$, there is $c>0$ such that if $r, s$ are positive integers with $s \le (1-1/q)r$ and $j = (1-1/q)r-s+1$, then 
\[
R(q+1;r,s) \le \max\left((1+\epsilon)A_q(r,s-cj),\epsilon s\right).
\]
Furthermore, if $q=p^i$ and $r=p^j$ are powers of a prime $p$ with $r \geq q$, then $R(q+1;r,s)\le (1+\epsilon)A_q(r,s-cj)$. 
\end{thm}

We suspect that there may even be equality in Theorem~\ref{thm:ramcode1} when $s$ is sufficiently close to $(1-1/q)r$, though our methods fall somewhat short of proving this.

Having established this connection, we can use it to prove bounds on $R(q+1; r, s)$ when $s$ is close to $(1-1/q)r$ by studying the bounds for $A_q(r,s)$ in the same range. It turns out that the study of such bounds is a well-established topic in coding theory, particularly in the binary case. We have already mentioned the work of Plotkin above. More precisely, he showed that $A_{2}(r,r/2)\le2r$ and that $A_{2}(r,s)\le2\lfloor s/(2s-r)\rfloor$
for $s>r/2$, both of which are sometimes tight by considering Hadamard codes (see, for instance, \cite[Chapter
2]{MS}). 
More generally, Blake and Mullin \cite{BM} showed
that $A_{q}(r,s)\le\frac{qs}{qs-r(q-1)}$ when $s>(1-1/q)r$ and
it can also be shown that $A_{q}(r,(1-1/q)r) \le 2qr$. 

There has also been a great deal of work in the binary case for $s$ of the form $(r-j)/2$ (see Figure~\ref{fig:summary}).
For instance, using the linear programming bound, McEliece (see \cite[Chapter 17]{MS}) showed that $A_{2}(r,(r-j)/2)\le(1+o(1))r(j+2)$
for $j=o(r^{1/2})$. Sidel'nikov \cite{Si} constructed a code 
showing that McEliece's bound is asymptotically tight when $j=\Theta(r^{1/3})$.
In particular, he showed that $A_{2}(r,(r-j)/2)\ge r(j+2)+1$ for $r=(2^{4m}-1)/(2^{m}+1)$
and $j=2^{m}-1$. 
Later, Tiet\"{a}v\"{a}inen~\cite{T} (see also~\cite{KL}) showed that $A_{2}(r,(r-j)/2) = O(r\log(j+1))$
for $j=o(r^{1/3})$ and conjectured that $A_{2}(r,(r-j)/2)=O(r)$
in this range. Very recently, this conjecture was resolved in a strong form by Balla~\cite{Balla}, who showed that $A_{2}(r,(r-j)/2) \leq (2+o(1))r$
for $j=o(r^{1/3})$. That is, the bound remains close to the Plotkin bound in this range. 

\begin{figure}
\begin{center}

\begin{tabular}{|r|r|r|}
\hline
$j$&$A_2(r, (r-j)/2)$&Authors\\ 
\hline
$0$&$\le 2r$&Plotkin~\cite{Plot}\\
$o(r^{1/3})$&$O(r \log(j+1))$&Tiet\"av\"ainen~\cite{T}\\ 
$o(r^{1/3})$&$\le (2+o(1))r$&Balla~\cite{Balla}\\
$\Theta(r^{1/3})$&$\Theta(r^{4/3})$&Sidel'nikov~\cite{Si}\\
$o(r^{1/2})$&$\leq (1+o(1))r(j+2)$&McEliece~\cite{MS}\\
$2^c r^{1/2}$&$\geq r^c$&Pang et al.~\cite{PMP}\\
$\leq \sqrt{(k-1)r}$&$O(r^k)$&This paper\\
$\geq (k-1)\sqrt{r/2}$&$ \ge (rq)^{k/2}$&This paper\\
\hline
\end{tabular}
\end{center}
\caption{A summary of the known bounds for $A_2(r, (r-j)/2)$.} \label{fig:summary}
\end{figure}

In a recent paper, Pang, Mahdavifar and Pradhan
\cite{PMP} showed that $A_{2}(r,(r-2^{c}r^{1/2})/2)\ge r^{c}$ and that 
$A_{2}(r,(r-2\sqrt{r})/2)=O(r^{7/2})$ and $A_{2}(r,(r-4\sqrt{r})/2)=O(r^{15/2})$. 
We improve these bounds and, more generally, establish good bounds for 
$A_q(r, (1-1/q)(r-j))$ when $j$ is on the order of $\sqrt{r}$. Moreover, because of Theorems~\ref{thm:ramcode1} and~\ref{thm:ramcode2},
we get analogues, both of the bounds here and those mentioned above, for the corresponding set-coloring Ramsey numbers $R(q+1; r, s)$.

\begin{thm} \label{thm:codebounds}
If $k$ is a positive integer and $j\le\sqrt{(k-1)r/(q-1)}$, then
\[
A_{q}(r,(1-1/q)(r-j))=O_{q,k}(r^{k}).
\]
On the other hand, for any prime power $q$, there are infinitely many $r$ such that, for $j\ge(k-1)\sqrt{r/q}$,
\[
A_{q}(r,(1-1/q)(r-j))\ge(rq)^{k/2}.
\]
\end{thm}

As a warm-up to our main result, in the next section we will prove a tight result for $R(3; 2s, s)$ (see also~\cite[Proposition 4.3]{CFHMSV} for another tight result). This quantity was recently studied, independently of the work in~\cite{CFHMSV}, in the master's thesis of Le~\cite{Le}. She showed that if there is a Hadamard matrix of order $2s$, then $R(3;2s,s) \geq 4s+1$. In the other direction, she gave an upper bound on $R(3;2s,s)$ which grows exponentially in $s$ and asked whether the gap can be closed. We answer this question by proving the following.

\begin{thm}\label{generalupperbound}
For all $s>1$, $R(3;2s,s) \leq 4s+1$. 
\end{thm}

Note that the assumption $s>1$ is needed in Theorem \ref{generalupperbound} as $R(3;2,1)=R(3;2)=6$. 
Moreover, since there is a Hadamard matrix of order $2s$ whenever $s = q + 1$ with $q \equiv 1 \pmod 4$ a prime power, 
we see that $R(3;2s,s)=4s+1$ for infinitely many $s$ and also that $R(3;2s,s) = (4+o(1))s$. 

\section{A tight result infinitely often}

In this short section, we prove Theorem~\ref{generalupperbound}, that $R(3;2s,s) \leq 4s+1$ for all $s > 1$, which, by Le's construction~\cite{Le}, is sharp for infinitely many $s$. We begin with the following result, which is essentially a special case of~\cite[Proposition 4.1]{CFHMSV}. 

\begin{lem}\label{easy}
If $r<2s$, then $R(3;r,s) \leq \frac{2s}{2s-r}+1$. In particular, $R(3;2s-1,s) \leq 2s+1$.  
\end{lem}

\begin{proof}
Consider an $(r,s)$-coloring of the edges of the complete graph on $N$ vertices with no monochromatic triangle. 
As each of the $r$ color classes is triangle-free, each color class has at most $N^2/4$ edges, so the total number of colors used on all edges is at most $rN^2/4$. On the other hand, as $s$ colors are used on each edge, the total number of colors used on all edges is $s{N \choose 2}$. Hence, $s{N \choose 2} \leq rN^2/4$. Simplifying, 
we get that $1-1/N \leq r/2s$ and so $N \leq \frac{2s}{2s-r}$. Thus, $R(3;r,s) \leq \frac{2s}{2s-r}+1$. 
\end{proof}

In the proof above, we used Mantel's theorem, the statement that any triangle-free graph on $N$ vertices has at most $\lfloor N^2/4 \rfloor$ edges. It is known that equality holds in Mantel's theorem if and only if the graph is a balanced complete bipartite graph. If, instead, we restrict attention to non-bipartite graphs, Mantel's theorem can be improved very slightly. This is the content of the following result of Brouwer~\cite{Brouwer}.

\begin{lem}[\cite{Brouwer}] \label{lem:Brouwer}
Any non-bipartite triangle-free graph on $N$ vertices has at most $\lfloor N^2/4 \rfloor -\lfloor N/2 \rfloor +1$ edges. In particular, when $N$ is odd, any such graph has at most $N^2/4 - N/2+5/4$ edges. 
\end{lem}


With this, we can now prove Theorem~\ref{generalupperbound}.

\begin{proof}[Proof of Theorem \ref{generalupperbound}]
Consider a $(2s,s)$-coloring of the edges of the complete graph on $N=4s+1$ vertices and suppose, for the sake of contradiction, that it has no monochromatic triangle. If one of the color classes 
has an independent set $S$ of size $2s+1$, then the coloring induced on the set $S$ is a $(2s-1,s)$-coloring and so, by Lemma \ref{easy}, $S$ must contain a monochromatic triangle, a contradiction. Since any bipartite graph on $4s+1$ vertices contains an independent set with $2s+1$ vertices, to complete the proof it suffices to show that at least one of the color classes is bipartite.  
But, if each color class is non-bipartite, Lemma~\ref{lem:Brouwer} implies that each color class has at most $N^2/4 - N/2+5/4$ edges, so the total number of colors on edges is at most $2s(N^2/4 - N/2+5/4)$. As the total number of colors on edges equals $s{N \choose 2}$, we would then obtain $s{N \choose 2} \leq 2s(N^2/4 - N/2+5/4)$. This simplifies to $N \leq 5$ or, equivalently, $s \leq 1$, contradicting our assumption that $s>1$ and completing the proof.  
\end{proof}

As a quick corollary of Theorem~\ref{generalupperbound}, applied in combination with Theorem~\ref{thm:ramcode1}, we see that $A_2(2s, s) \leq R(3;2s,s)-1 \leq 4s$, which is exactly the Plotkin bound in the binary case. As the Plotkin bound is known to be tight whenever there is a Hadamard matrix of order $2s$, this also returns Le's lower bound~\cite{Le} for $R(3; 2s, s)$.

\section{Codes from set colorings}

In this section, we prove Theorem~\ref{thm:ramcode2}, showing that the connection between codes and set-coloring Ramsey numbers discovered in~\cite{CFHMSV} goes both ways near the zero-rate threshold. We first state and prove a certain stability version of Tur\'an's theorem.

\subsection{Stability for Tur\'an's theorem}

Tur\'an's theorem is the natural generalization of Mantel's theorem to larger cliques. If we write $T_{N,q}$ for the Tur\'an graph, the balanced complete $q$-partite graph on $N$ vertices, Tur\'an's theorem \cite{Turan} then states that the Tur\'an graph $T_{N,q}$ is the unique $K_{q+1}$-free graph on $N$ vertices with the maximum number of edges. This maximum is therefore at most $(1-\frac{1}{q})N^2/2$ edges, with equality if and only if $N$ is a multiple of $q$. 

We wish to prove a stability version of Tur\'an's theorem, saying that any graph on $N$ vertices with nearly as many edges as $T_{N,q}$ can be made $q$-partite by deleting a small number of vertices. In the proof, we will make use of the following well-known result of Andr\'asfai, Erd\H{o}s and S\'os~\cite{AES}.

\begin{lem}[\cite{AES}] \label{lem:AES}
Every $K_{q+1}$-free graph on $N$ vertices with minimum degree larger than $\frac{3q-4}{3q-1}N = (1-\frac{1}{q - 1/3})N$ is $q$-partite.
\end{lem}

The stability result we need is now as follows.

\begin{lem}\label{turanstability}
Every $K_{q+1}$-free graph $G$ on $N \ge 12 q^2$ vertices has at most $(1-\frac{1}{q})N^2/2 -\frac{Nf_q(G)}{8q^2}$ edges, where $f_q(G)$ is the minimum $f$ such that $f$ vertices can be deleted from $G$ so that the remaining induced subgraph is $q$-colorable.
\end{lem}

\begin{proof}
Let $G(0)=G$. After defining $G(i)$, if $G(i)$ has a vertex $v_i$ of degree at most $\frac{3q-4}{3q-1}|G(i)|$, then let $G(i+1)$ be obtained from $G(i)$ by deleting $v_i$. Let $f=f_q(G)$. We must eventually define $G(f)$, as otherwise the process stops at some $G(i)$ with $i<f$ of minimum degree larger than $\frac{3q-4}{3q-1}|G(i)|$. But, by Lemma~\ref{lem:AES}, this $G(i)$ is a $q$-partite graph obtained from $G$ by deleting $i<f=f_q(G)$ vertices, contradicting the definition of $f_q(G)$. 

Since $G(f)$ is $K_{q+1}$-free, Tur\'an's theorem implies that $G(f)$ has at most $(1-\frac{1}{q})|G(f)|^2/2$ edges. Hence, since the degree of $v_i$ in $G(i)$ is at most $\frac{3q-4}{3q-1}(N-i)$ and $\frac{3q-4}{3q-1} = (1-\frac{1}{q}) - \frac{1}{q(3q-1)}$, the number of edges in $G$ is at most
$$e(G(f))+\sum_{i=0}^{f-1} \frac{3q-4}{3q-1}(N-i) \leq \left(1-\frac{1}{q}\right)N^2/2 + \frac f2 -\frac{Nf}{2q(3q-1)} \leq \left(1-\frac{1}{q}\right)N^2/2 -\frac{Nf}{8q^2},$$
as required.
\end{proof}

\subsection{From set colorings to error-correcting codes}

We will deduce Theorem~\ref{thm:ramcode2} from the following result.

\begin{thm} \label{thm:setscodes2}
Let $\lambda>1$ and $N=R(q+1;r,s)-1 \geq 12 q^2$. If $b=\left \lfloor 4\lambda q^2\left(\left(1-\frac{1}{q}\right)r-s+\frac{s}{N}\right) \right \rfloor \geq 0$, then 
$$A_q(r,s-2b) \geq \left(1-\frac{1}{\lambda}\right)N.$$ 
\end{thm}

\begin{proof}
Consider an $(r,s)$-coloring of $K_N$ with $N=R(q+1;r,s)-1$ without a monochromatic $K_{q+1}$. Such a coloring exists from the definition of the set-coloring Ramsey number. Consider the $r$ graphs $G_1,\ldots,G_r$ on $V(K_N)$ where $E(G_i)$ is the set of edges of $K_N$ whose set of colors contains color $i$, so that $G_i$ is $K_{q+1}$-free and each edge of $K_N$ is an edge of exactly $s$ of these $r$ graphs. For each $G_i$, there is a set $U_i$ of $f_q(G_i)$ vertices such that the induced subgraph of $G_i$ upon deleting $U_i$ is $q$-partite. 
If we write $V_{i1},\ldots,V_{iq}$ for the $q$ resulting independent sets in $G_i$, then $V(K_N)$ can be written as the disjoint union $V(K_N)=U_i \sqcup V_{i1} \sqcup \dots \sqcup V_{iq}$. For each vertex $v \in V(K_N)$, let $x_i(v)=j$ if $v \in V_{ij}$ and otherwise let $x_i(v)$ be an arbitrary element of $[q]$. Then, for each $v \in V(K_N)$, we let $x(v)=(x_1(v),\ldots,x_r(v)) \in [q]^r$.

By counting over each edge $e$, the number of pairs $(e,i)$ with $e \in E(G_i)$ is ${N \choose 2}s$. On the other hand, counting over the color classes, the number of such pairs is also $\sum_{i=1}^r e(G_i)$. Hence, 
\[{N \choose 2}s = \sum_{i=1}^r e(G_i) \leq \sum_{i=1}^r \left(\left(1-\frac{1}{q}\right)N^2/2 -\frac{Nf_q(G_i)}{8q^2}\right) = \left(1-\frac{1}{q}\right)rN^2/2 - \frac{N}{8q^2}\sum_{i=1}^{r}f_q(G_i),\]
where the inequality is by Lemma \ref{turanstability}. Multiplying both sides by $\frac{8q^2}{N^2}$ and rearranging, we get 
$$N^{-1}\sum_{i=1}^r f_q(G_i) \leq 4q^2\left(\left(1-\frac{1}{q}\right)r-s+\frac{s}{N}\right):=M.$$
Since $\sum_{i=1}^r |U_i| = \sum_{i=1}^r f_q(G_i)$, 
Markov's inequality now implies that the number of vertices $v$ for which $v \in U_i$ for at least $\lambda M$ values of $i$ is at most $N/\lambda$. Hence, the set $V'$ of vertices $v$ for which $v \in U_i$ for at most $\lambda M$ values of $i$ satisfies $|V'| \geq N - N/\lambda = (1-\lambda^{-1})N$. Consider the collection of codewords $C=\{x(v):v \in V'\}$. For each pair of distinct vertices $u,v \in V'$, we have that $(u,v)$ is an edge of exactly $s$ graphs $G_i$. For each $G_i$ for which $(u,v) \in E(G_i)$ and neither $u$ nor $v$ is in $U_i$, we have $x_i(u) \not = x_i(v)$. Since $u$ and $v$ are each in at most $b = \lfloor \lambda M \rfloor$ of the sets $U_i$, there are at least $s-2b$ coordinates for which $u$ and $v$ must differ. Hence, since $C$ is a collection of codewords in $[q]^r$ in which each pair has distance at least $s-2b$, $|C| \leq A_q(r,s-2b)$. Since $|C| = |V'| \geq (1-\lambda^{-1})N$, this completes the proof.
\end{proof}

\begin{proof}[Proof of Theorem~\ref{thm:ramcode2}]
    If $N=R(q+1;r,s)-1 < \epsilon s$, then we are already done. We may therefore assume that $N \ge \epsilon s$. We apply Theorem \ref{thm:setscodes2} with $\lambda = 2/\epsilon$ to obtain 
    \[
    A_q(r,s-2b) \ge (1-\epsilon/2)N,
    \]
    where $b=\left \lfloor 4\lambda q^2\left(\left(1-\frac{1}{q}\right)r-s+\frac{s}{N}\right)\right\rfloor$. Note that $b \le cj/2$ where $j=\left(1-\frac{1}{q}\right)r-s+1$ for an appropriate constant $c>0$ depending only on $q$ and $\epsilon$. This implies that 
    \[
    R(q+1;r,s) \le (1+\epsilon)A_q(r,s-cj),
    \]
    as desired. 

    In the case where $q=p^i$ and $r=p^j$ are powers of the same prime $p$ with $r \geq q$, we show that $N > s$, which immediately gives the desired conclusion. Based on generalized Hadamard matrices, it is shown in \cite{MkS} that for such $q$ and $r$ there are codes over $\mathbb{F}_{q}^{r}$ with size $qr$ and distance $(1-1/q)r$. By Theorem \ref{thm:ramcode1}, this implies that $N \ge A_q(r,s) \ge qr > s$, as required.
\end{proof}

\section{Codes with large distance}

In this section, we prove our upper and lower bounds for $A_q(r,s)$, and hence $R(q+1; r,s)$, when $s$ is close to $(1-1/q)r$. For the upper bound, we will make use of Delsarte's linear programming bound~\cite{Del}, following a technique of McEliece, Rodemich, Rumsey and Welch~\cite{MRRW77} (see also Theorem 35 in \cite[Chapter 17]{MS}) and its extension to $q$-ary codes in \cite{A79}. If we define the Krawtchouk polynomials by
\[
K_{i}^{q,r}(x)=\sum_{j=0}^{i}(-1)^{j}(q-1)^{i-j}\binom{x}{j}\binom{r-x}{i-j}
\]
for any $0 \leq i \leq r$, then Delsarte's bound is as follows. 

\begin{lem}{\cite{Del}} \label{lem:Del}
If
$P(x)=\sum_{i}\beta_{i}K_{i}^{q,r}(x)$ is a linear combination of the $K_{i}^{q,r}$ with $\beta_0 > 0$ and $\beta_{i}\ge0$ for all
$i\ge1$ such that $P(d)\le0$ for all $D \le d\le r$, then 
\[
A_{q}(r,D)\le P(0)/\beta_{0}.
\]
\end{lem}

We are now ready to prove the upper bound in Theorem~\ref{thm:codebounds}. 

\begin{thm}
If $k$ is a positive integer and  $j\le\sqrt{(k-1)r/(q-1)}$, then
\[
A_{q}(r,(1-1/q)(r-j))=O_{q,k}(r^{k}).
\]
\end{thm}

\begin{proof}
Our argument will largely follow the proof from~\cite{MRRW77}. We refer to this paper and to~\cite{Lev} for the standard properties of Krawtchouk polynomials. Note that throughout the proof, for clarity of presentation, we will systematically omit the superscripts in the notation for Krawtchouk polynomials.

For $a < (1-1/q)(r-j)$, consider the polynomial 
\[
P(x)=\frac{(K_{i}(a)K_{i+1}(x)-K_{i+1}(a)K_{i}(x))^{2}}{a-x},
\]
noting that $P(d)\le0$ for all $d\ge (1-1/q)(r-j)$. 
By the Christoffel--Darboux
formula (see~\cite[Corollary 3.5]{Lev})
\[
\frac{K_{i}(a)K_{i+1}(x)-K_{i+1}(a)K_{i}(x)}{x-a}=-\frac{q}{i+1}\binom{r}{i}(q-1)^{i}\sum_{j=0}^{i}\frac{K_{j}(x)K_{j}(a)}{\binom{r}{j}(q-1)^{j}},
\]
we have
\begin{align*}
&P(x) =\frac{q}{i+1}\binom{r}{i}(q-1)^{i}(K_{i}(a)K_{i+1}(x)-K_{i+1}(a)K_{i}(x))\sum_{j=0}^{i}\frac{K_{j}(x)K_{j}(a)}{\binom{r}{j}(q-1)^{j}}\\
 & =-\frac{q}{i+1}\binom{r}{i}(q-1)^{i}\sum_{j=0}^{i}K_{j}(x)K_{i}(x)\cdot\frac{K_{j}(a)K_{i+1}(a)}{\binom{r}{j}(q-1)^{j}}+\frac{q}{i+1}\binom{r}{i}(q-1)^{i}\sum_{j=0}^{i}K_{j}(x)K_{i+1}(x)\cdot\frac{K_{j}(a)K_{i}(a)}{\binom{r}{j}(q-1)^{j}}.
\end{align*}
We will make use of the following properties of Krawtchouk polynomials: 
$K_{i}(x)K_{j}(x)$ is a nonnegative combination of the $K_{\ell}(x)$;\footnote{This should be taken as meaning that the values of the two polynomials are equal for all $x = 0, 1, \dots, r$, but this is sufficient for our application of Delsarte's bound.}
the $K_{\ell}$ are orthogonal under the bilinear form $\langle f,g\rangle=\sum_{j=0}^{r}\binom{r}{j}(q-1)^{j}f(j)g(j)$ with $\langle K_i, K_i\rangle= q^{r}(q-1)^{i}\binom{r}{i}$;
and if $\rho_{i}$ is the smallest positive root of $K_{i}$,
then $\rho_{i}>\rho_{i+1}$ and there are no other roots of $K_{i+1}$ in $(\rho_{i+1}, \rho_i)$. We also have 
\[
\beta_{0}=q^{-r}\sum_{x=0}^{r}\binom{r}{x}(q-1)^{x}P(x),\qquad K_{i}(0)=(q-1)^{i}\binom{r}{i}.
\]

If $a$ is such that $\rho_{i+1}<a<\rho_{i}$, then 
$K_{j}(a)K_{i+1}(a)\le0$ and $K_{j}(a)K_{i}(a)\ge0$ for all $j\le i$. Therefore, $P(x)=\sum_{i}\beta_{i}K_{i}^{q,r}(x)$ is a linear combination of the $K_{i}^{q,r}$ with $\beta_{i}\ge0$ for all
$i\ge1$. 
Moreover, by orthogonality, we have that
\begin{align*}
\beta_{0} & =q^{-r}\sum_{x=0}^{r}\binom{r}{x}(q-1)^{x}P(x)=-q^{-r}\frac{q}{i+1}\binom{r}{i}(q-1)^{i}\cdot\frac{K_{i}(a)K_{i+1}(a)}{\binom{r}{i}(q-1)^{i}}\cdot q^{r}(q-1)^{i}\binom{r}{i}\\
 & =-\frac{q}{i+1}(q-1)^{i}\binom{r}{i}K_{i}(a)K_{i+1}(a)
\end{align*}
and 
\[
P(0)=\frac{(K_{i}(a)K_{i+1}(0)-K_{i+1}(a)K_{i}(0))^{2}}{a}.
\]
Hence, if $\rho_{i+1}<a<\rho_{i}$, Lemma~\ref{lem:Del} implies that 
\begin{align*}
A_{q}(r,(1-1/q)(r-j)) & \le P(0)/\beta_{0}=-\frac{(K_{i}(a)K_{i+1}(0)-K_{i+1}(a)K_{i}(0))^{2}}{a\frac{q}{i+1}(q-1)^{i}\binom{r}{i}K_{i}(a)K_{i+1}(a)}.
\end{align*}
Noting that $K_{i+1}(a)/K_{i}(a)$ ranges from $0$ to $-\infty$
as $a$ goes from $\rho_{i+1}$ to $\rho_{i}$, we can find $a_0$ in that range such
that $K_{i+1}(a_0)/K_{i}(a_0) = - K_{i+1}(0)/K_{i}(0)$ and
\[
-\frac{(K_{i}(a_0)K_{i+1}(0)-K_{i+1}(a_0)K_{i}(0))^{2}}{a_0\frac{q}{i+1}(q-1)^{i}\binom{r}{i}K_{i}(a_0)K_{i+1}(a_0)}=\frac{4 K_i(0) K_{i+1}(0)}{a_0\frac{q}{i+1}(q-1)^{i}\binom{r}{i}}
=\frac{4(q-1)^{i+1}\binom{r}{i+1}}{a_0\frac{q}{i+1}}\le\frac{4(i+1)(q-1)^{i+1}\binom{r}{i+1}}{q\rho_{i+1}}.
\]

Now we note that $K_{1}(x)=(q-1)(r-x)-x$, so that $\rho_{1}=(1-1/q)r$ and
$K_{2}(x)=(q-1)^{2}\binom{r-x}{2}-(q-1)x(r-x)+\binom{x}{2}$, so that $\rho_{2} \leq (1-1/q)(r-\sqrt{r/(q-1)})$.
Writing $h_k$ for the largest root of the $k$-th Hermite polynomial, we also have the general bound (see \cite[Corollary 6.1]{Lev})
\[
\rho_{k}\le (1-1/q)r-\frac{q-2}{2q}h_{k}^{2}-\frac{\sqrt{2(q-1)(r-k+2)}h_{k}}{q} \leq (1-1/q)(r - \sqrt{(k-1)r/(q-1)}),
\]
where we used that $h_{k}>\sqrt{(k-1)/2}$ for $k>2$ and that $r$ may be taken sufficiently large in $q$ and $k$. 
For $j \leq \sqrt{(k-1)r/(q-1)}$, we may therefore pick any $a$ with $\rho_{k+1} < a < \rho_{k}$  and it will automatically satisfy $a < (1-1/q)(r-j)$. Therefore, taking $a = a_0$ as in the calculation above, we find that 
\[A_q(r, (1-1/q)(r-j)) \leq \frac{4(k+1)(q-1)^{k+1}\binom{r}{k+1}}{q\rho_{k+1}} = O_{q,k}(r^k),\]
where we used that $\rho_{k+1} = \Omega_{q,k}(r)$ (see~\cite[Equation 125]{Lev}).
\end{proof}

We now prove the lower bound in Theorem~\ref{thm:codebounds}, which follows from concatenating appropriate codes.

\begin{thm}
For any positive integer $k$ and any prime power $q$, there are infinitely many $r$ such that, for $j\ge(k-1)\sqrt{r/q}$,
\[
A_{q}(r,(1-1/q)(r-j))\ge(rq)^{k/2}.
\]
\end{thm}

\begin{proof}
Based on generalized Hadamard matrices, it is shown in \cite{MkS}
that if $q = p^i$ and $u = p^j$ for some $j \geq i$, then there exist codes over $\mathbb{F}_{q}^{u}$ with size $qu$
and distance $(1-1/q)u$. We also recall that the Reed--Solomon code is
a code over $\mathbb{F}_{s}^{n}$ with size $s^{k}$ and distance
$n-k+1$, where $s$ is a prime power at least $n$. 

We consider a concatenation code with the generalized Hadamard
code as the inner code and the Reed--Solomon code as the outer code.
More explicitly, let ${\cal C}_{i}$ be a code over $\mathbb{F}_{q}^{u}$
with size $qu$ and distance $(1-1/q)u$ and let ${\cal C}_{o}\subseteq\mathbb{F}_{s}^{n}$
be the Reed--Solomon code where we choose $s = n = qu$ to be a prime power. 
The concatenation code is formed by  considering ${\cal C}_{o}$ as a subset of $[qu]^{n}$ through a bijection $\phi:[s]\to qu$ and then using the inner code ${\cal C}_{i}$ to map each element of $[qu]^n$ term by term to a subset of $(\mathbb{F}_{q}^{u})^n = \mathbb{F}_{q}^{un}$. Since it is easy to see that the distance of a concatenated code is at least the product of the distances of the inner and outer codes, this gives
a code in $\mathbb{F}_{q}^{un}$ with distance at least $(n-k+1)(1-1/q)u$
and size $s^{k}$. Letting $r=un=qu^{2}$, we see that the distance of the
code is at least $(1-1/q)(r-(k-1)\sqrt{r/q})$ and the size of the
code is $(rq)^{k/2}$. 
\end{proof}

\section{Concluding remarks}

Using Theorems~\ref{thm:ramcode1} and \ref{thm:ramcode2} to turn back to $R(3; r, (r-j)/2)$, the picture that emerges is a rather complex one, with the function exhibiting a range of different behaviours depending on the value of $j$. When $r$ is even and $j = 0$, Theorem~\ref{generalupperbound} gives the exact value $R(3;r,r/2) = 2r + 1$. Increasing $j$, the result of Balla~\cite{Balla} discussed in the introduction tells us that $R(3; r, (r-j)/2)$ remains close to $2r$ until $j$ reaches roughly $r^{1/3}$, where the result of Sidel'nikov \cite{Si} shows that the value jumps to $r^{1 + \epsilon}$ for some $\epsilon > 0$. The function is at most roughly $rj$ until $j$ passes $\sqrt{r}$, where the results of Pang, Mahdavifar and Pradhan~\cite{PMP}, which our Theorem~\ref{thm:codebounds} refines, show that the function starts to grow as an arbitrary power of $r$. By the time $j$ is linear in $r$, the bound becomes exponential in $r$ and it is possible (though not generally expected) that the bound jumps to superexponential as $j$ approaches $r$. 

This summary raises many questions, not least of which is whether there are further jumps in behaviour as $j$ passes from $\sqrt{r}$ to $r$. It would also be interesting to decide whether any aspects of the picture drawn above change as the clique size goes from $3$ to $4$ and beyond. For instance, does the shift from the linear to the polynomially superlinear regime for $R(4; r, 2(r-j)/3)$ still happen when $j$ is roughly $r^{1/3}$?

\vspace{3mm}
\noindent
{\bf Acknowledgements.} 
We would like to thank Xiaoyu He, Dhruv Mubayi, Andrew Suk and Jacques Verstra\"ete for helpful conversations. We are also grateful to Maria Axenovich for bringing the work of Yen Hoang Le to our attention.

\end{document}